\documentclass[11pt,a4paper]{article}
\usepackage{a4wide,amsmath,amsthm,amssymb,amsfonts,color}
\usepackage{mathtools}
\usepackage{hyperref} 
\mathtoolsset{showonlyrefs}
\usepackage[all]{xy}
\usepackage[utf8]{inputenc}
\usepackage{enumitem,MnSymbol,authblk}

\usepackage{tikz,tikz-3dplot}
\usepackage{mathrsfs}
\usepackage{cancel}

\newtheorem{thm}{Theorem}[section]
\newtheorem{prop}[thm]{Proposition}

\newtheorem{dfn}[thm]{Definition}

\theoremstyle{remark}
\newtheorem{rmkk}[thm]{Remark}
\newtheorem{exe}[thm]{Example}
\newenvironment{rmk}{\begin{rmkk}\rm}{\qee\end{rmkk}}

\newcommand{\qee}{\mbox{\hspace{0.2mm}}\hfill$\triangle$}

\newcommand{\C}{\mathbb{C}}
\newcommand{\Q}{\mathbb{Q}}
\newcommand{\Pj}{\mathbb{P}}

\newcommand{\N}{\mathbb{N}}
\newcommand{\cO}{\mathcal{O}}

\newcommand{\codim}{\operatorname{codim}}

\newcommand{\Pic}{\operatorname{Pic}}

\title{\bf \large     CODIMENSION BOUNDS     FOR THE  NOETHER-LEFSCHETZ \\[5pt]   COMPONENTS
  FOR TORIC  VARIETIES}
\author[$^\S$]{Ugo Bruzzo}
\author[$\ddag$]{William D. Montoya}
\affil[$\S \ddag$]{\,SISSA (Scuola Internazionale Superiore di Studi Avanzati),  \par\vskip-3pt Via Bonomea 265, 34136 Trieste, Italy  }
\affil[$^\S$]{Departamento de Matem\'atica, Universidad Federal da Para\'iba, \par\vskip-3pt Campus I, Jo\~ao Pessoa, PB, Brazil}
\affil[$^\S$]{INFN (Istituto Nazionale di Fisica Nucleare), Sezione di Trieste}
\affil[$^\S$]{IGAP (Institute for Geometry and Physics), Trieste}
\affil[$^\S$]{Arnold-Regge Center for Algebra, Geometry \par\vskip-3pt and Theoretical Physics, Torino}
\affil[$\ddag$]{ICTP (International Centre for Theoretical Physics), \par\vskip-3pt Strada Costiera, 11, 34151 Trieste, Italy}
\date{}
\begin{document}
\maketitle

\begin{abstract} 
 For a quasi-smooth hyper-surface $X=\{f=0\}$ in a projective simplicial toric variety $\Pj_{\Sigma}$, the morphism  $i^*:H^p(\Pj_{\Sigma})\to  H^p(X)$ induced by the inclusion,  is injective for $p=d$ and an isomorphism for $p<d-1$, where $d=\dim \Pj_{\Sigma}$. This allows one to define the Noether-Lefschetz locus
 $NL_{\beta}$ as the locus   of quasi-smooth hypersurfaces of degree $\beta$ such that $i^*$ acting on
 the middle algebraic cohomology is not an isomorphism.  In this paper we prove that, under some assumptions,  if $\dim\Pj_{\Sigma}=2k+1$
 and 
 $k\beta-\beta_0=n\eta$ $(n\in\N)$,  where  $\eta$ is the class of an ample divisor, then every irreducible component $V$ of the Noether-Lefschetz locus  quasi-smooth hypersurfaces of degree $\beta$ satifies the bounds 
$$n+1\leq \codim V \leq h^{k-1,k+1}(X).$$
\end{abstract}

 \let\svthefootnote\thefootnote
\let\thefootnote\relax\footnote{
\hskip-2\parindent 
Date:  \today  
Email: {\tt  bruzzo@sissa.it, wmontoya@sissa.it}
}
\addtocounter{footnote}{-1}\let\thefootnote\svthefootnote

\newpage

\section{Introduction}

The classical Noether-Lefschetz theory is about the Picard number of surfaces in 3-dimensional projective space. Let
$\mathcal U_d \subset \Pj H^0(\Pj^3,\cO_{\Pj^3}(d))$ be the locus of smooth surfaces of degree $d$ in $\Pj^3$, with $d\ge 4$; then
the very general surface in $\mathcal U_d$ has Picard number 1 (for an historical perspective of the Noether-Lefschetz problem, and exhaustive references the reader may consult \cite{bn}) .  Moreover, if
$Z$ is a component of the locus in $\mathcal U_d$ whose points correspond to surfaces with Picard number greater than 1 (the Noether-Lefschetz locus), then
$$ d - 3 \le \codim_{\mathcal U_d} Z \le {d-1 \choose 3} .$$
This result was generalized in \cite{bg1,LanMar} to quasi-smooth surfaces\footnote{A neat way to define the notion of quasi-smooth hypersurface $X$  in a toric variety $\Pj_\Sigma$ is to regard $\Pj_\Sigma$ as an orbifold: then $X$ is quasi-smooth if and only if it is a sub-orbifold of
$\Pj_\Sigma$. Heuristically, $X$ is quasi-smooth if its only singularities are those ``inherited'' from $\Pj_\Sigma$.}
 in projective simplicial toric threefolds satisfying some conditions.
The purpose of the present paper is to extend these bounds to the case of projective simplicial toric  varieties of higher odd dimension, see Theorems \ref{lowercodim} and \ref{uppercodim} (when the ambient variety has even dimension
the problem is trivial as the middle cohomology of hypersurfaces is controlled by the Lefschetz hyperplane theorem).

This short  paper is a natural sequel to  \cite{bm}, where the definition of the Noether-Lefschetz locus was extended
to simplicial projective toric varieties $ \Pj^{2k+1}_{\Sigma}$ of arbitrary odd dimension.
Given an ample class $\beta$  in $\Pic(\Pj^{2k+1}_{\Sigma})$, one considers
sections $f\in \Pj(H^0(\cO_{\Pj^{2k+1}_{\Sigma}}(\beta))) $  
such that $X_f=\{f=0\}$ is a quasi-smooth hypersurface.
Let $\mathcal{U}_{\beta}\subset \Pj(H^0(\cO_{\Pj^{2k+1}_{\Sigma}}(\beta))) $ be the open subset parameterizing   quasi-smooth hypersurfaces and let $\pi: \chi_{\beta}\to  \mathcal{U}_{\beta}$ be the tautological family. One considers
the local system $\mathcal{H}^{2k} = R^{2k}\pi_{\star}\C \otimes \cO_{\mathcal{U}_{\beta}}$ over $\mathcal{U}_{\beta}$.
The associated flat connection (the Gauss-Manin connection) will be denoted by $\nabla$. 

Let $0\neq \lambda_f\in H^{k,k}(X_f,\Q)/i^*(H^{k,k}(\Pj^{2k+1}_{\Sigma}))$ and let $U$ be a contractible open subset around $f$.  Finally, let $\lambda\in \mathcal{H}^{2k}(U)$ be the section defined by $\lambda_f$ and let $\bar{\lambda}$ its image in $(\mathcal{H}^{2k}/F^k\mathcal{H}^{2k})(U)$, where $ F^k\mathcal{H}^{2k} =\mathcal{H}^{2k,0}\oplus \mathcal{H}^{2k-1,1}\oplus \dots \oplus \mathcal{H}^{k,k}$.

\begin{dfn}[Local Noether-Lefschetz Locus] $N^{k,\beta}_{\lambda,U} = \{G\in U \mid \bar{\lambda}_{G}=0\}$. \label{defNL}\end{dfn}

In this paper we continue the study of the Noether-Lefschetz locus and establish 
 lower and   upper bounds for the codimension of its components. In section 2 we obtain the lower bound,    which,  following the terminology in \cite{bn}, we call the ``explicit Noether-Lefschetz theorem for toric varieties.''  In section 3, using the  Hodge theory  for hypersurfaces in complete simplicial toric varieties,  and the orbifold structure of the quasi-smooth hyper-surfaces (see \cite{bc}), we establish the upper bound, extending the ideas in \cite{bg1}.
 
 \smallskip
 {\bf Acknowledgements.} We thank Paolo Aluffi for useful discussions. The results presented here stemmed
 from a research project initiated by the first author with Antonella Grassi. 
 This research was partly supported by PRIN ``Geometria delle variet\`a algebriche" and by
GNSAGA-INdAM. 

\section{Explicit Noether-Lefschetz theorem in toric varieties}
This section is a natural extension to higher dimensions of 
 the ideas developed in \cite{bg1,LanMar} for the case of threefolds. To this end there are   two points  to consider:

1.  Let $S = \oplus_\beta S_\beta$ be the Cox ring of the toric variety   $\Pj_{\Sigma}^3$ under consideration.
	In \cite{bg1,LanMar} the following assumption was made. Let $\beta$ and $\eta$ be ample classes in
	$\Pic(\Pj_{\Sigma}^3)$, with $\eta$ primitive and 0-regular  (in the sense of Castelnuovo regularity), and
	$\beta-\beta_0=n\eta$ for some $n\ge 0$, where $\beta_0$ is the anticanonical class   of $\Pj_{\Sigma}^{3}$. Then  one assumes  that  
 the multiplication map $S_{\beta}\otimes S_{n\eta}\to  S_{\beta+n\eta}$ is surjective;  this implies that a very general quasi-smooth surface of degree $\beta$ in $\Pj_{\Sigma}^3$  has the same Picard number as $\Pj_{\Sigma}^3$.
 In the higher dimensional case, if  we assume again the surjectivity of the multiplication map, using   Theorem 10.13 and  Proposition 13.7 in \cite{bc},
    and Lemma 3.7 in \cite{bg}, one proves that the  primitive cohomology of degree $2k$ 
    of   a very general quasi-smooth hypersurface of degree $\beta$ is zero. Of course we recover the result of \cite{Green1}   when $k=1$.

2. In \cite{bg,LanMar}  it alsowas  assumed that $H^1(\cO_{\Pj_{\Sigma}^3}(\beta-\eta))=H^{2}(\cO_{\Pj_{\Sigma}^3}(\beta-2\eta))=0$, which allowed one to conclude that a certain vector bundle is $1$-regular  with respect to $\eta$. We assume    
	$$H^q(\cO_{\Pj_{\Sigma}^{2k+1}}(\beta-q\eta))=0 \qquad \mbox{for} \quad 1 \le q \le 2k $$ and   will prove the same regularity for that  vector bundle.

The next Theorem establishes the lower bound for the codimension of the components of the Noether-Lefschetz locus.

\begin{thm}
Let $\Pj_{\Sigma}^{2k+1}$ be a Gorenstein projective simplicial toric variety, $\eta$ a 0-regular primitive ample Cartier class,   and $\beta$ a Cartier class such that $k\beta-\beta_0=n\eta$ $(n>0)$,  where $\beta_0$ is the anticanonical class   of $\Pj_{\Sigma}^{2k+1}$.  Assume that the multiplication morphism $S_{\beta}\otimes S_{n\eta}\to  S_{\beta+n\eta}$ is surjective, and that $H^q(\cO_{\Pj_{\Sigma}^{2k+1}}(\beta-q\eta))=0$ for $q=1,\dots,2k$; then 
$$ \codim Z \ge n+1 $$
for every irreducible component $Z$ of the Noether-Lefschetz locus $ NL_{\lambda, U}^{k,\beta}$.
\label{lowercodim}
\end{thm}

\begin{proof} The proof is a higher dimensional generalization of that in \cite{bg1} (which in turn largely
mimics the proof of \cite{Green2,Green1} for the case of $\mathbb P^3$), with the modification proposed in \cite{LanMar}.
We take a base point free linear system $W$ in $H^0(\cO_{\Pj_{\Sigma}^{2k+1}}(\beta))$ and a complete flag of linear subspaces 
$$W=W_c\subset W_{c-1} \subset \cdots \subset W_1\subset W_0=H^0(\cO_{\Pj_{\Sigma}^{2k+1}}(\beta)) .$$
Let $M_i$ be  the kernel of the surjective map $W_i\otimes \cO_{\Pj_{\Sigma}^{2k+1}}\to  \cO_{\Pj_{\Sigma}^{2k+1}}(\beta)$,
which is locally free. 
We have to prove that $M_0$ is 1-regular  with respect to $\eta$, i.e., that $H^q(M_0((1-q)\eta))=0$ for every positive $q$
(this is the regularity property we hinted at in the introduction). Taking cohomology from
\begin{equation} \label{M0}  0 \to M_0 \to W_0 \otimes  \cO_{\Pj_{\Sigma}^{2k+1}}\to  \cO_{\Pj_{\Sigma}^{2k+1}}(\beta) \to 0\end{equation}
 we get 
$$0\to  H^0(M_0)\to  H^0(W_0\otimes \cO_{\Pj_{\Sigma}^{2k+1}})\xrightarrow{\pi} H^0(\cO_{\Pj_{\Sigma}^{2k+1}}(\beta))  \to  H^1(M_0)\to  0\,; $$
as $\pi$ is surjective, $H^1(M_0)=0$. The vanishing of $H^q(M_0(1-q)\eta)=0$ for $1<q\leq 2k+1$ is obtained by induction,  tensoring the short exact sequence \eqref{M0} by $\cO_{\Pj_{\Sigma}^{2k+1}}((1-q)\eta)$, and considering the segment of the  long exact sequence of cohomology
\begin{multline} \cdots \to  H^{q-1}(\cO_{\Pj_{\Sigma}^{2k+1}}(\beta-(q-1)\eta))\to  H^q(M_0(-(q-1)\eta))  \\ \to  H^q(W_0\otimes \cO_{\Pj_{\Sigma}^{2k+1}}(-(q-1)\eta))\to  \cdots \end{multline}
where $H^{q-1}(\cO_{\Pj_{\Sigma}^{2k+1}}(\beta-(q-1)\eta))=0$ by  the inductive assumption, while      $$H^q(W_0\otimes \cO_{\Pj_{\Sigma}^{2k+1}}(-(q-1)\eta))=0$$ as $\eta$ is 0-regular.

The rest of the proof follows as in \cite{bg1,LanMar}.
\end{proof}
%
%

\section{Upper bound for the Codimension of the Noether-Lefschetz Components in Toric Varieties}

The Explicit Noether-Lefschetz Theorem has provided a  lower bound for the codimension of the Noether-Lefschetz components. Hodge theory in toric varieties will give us the upper bound.  
For a class $\beta$ as in the previous Section, let $f$ be a point in the Noether-Lefschetz locus,
let $X_f$ be the corresponding hypersurface in $\Pj_\Sigma^{2k+1}$, and let $\lambda$ be a class as in Definition \ref{defNL}.

\begin{thm} $\codim   Z \leq h^{k-1,k+1}(X_f)$ for every irreducible component $Z$ of the Noether-Lefschetz locus $ NL_{\lambda, U}^{k,\beta}$.
\label{uppercodim}
\end{thm}

This section is devoted to proving this theorem. Classically it is a consequence of  Griffiths' Transversality, which we want to extend  to the context of projective simplicial toric varieties.

%
%
%
%
%

\paragraph{Variations of Hodge Structure.}
The tautological family $\pi: \mathcal{X}_{\beta} \subset \mathcal{U}_{\beta}\times \Pj_{\Sigma}\to  \mathcal{U}_{\beta}$ is of finite type and separated since $\mathcal{X}_{\beta}$ and $\mathcal{U}_{\beta}$ are varieties. By Corollary 5.1 in \cite{ve} there exists a Zariski open set $\mathcal U \subset \mathcal{U}_{\beta}$ such that $\mathcal{X}=\pi^{-1}(\mathcal{U})\to  \mathcal{U}$ is a locally trivial fibration in the classical topology,  i.e., there exists an open cover of $\mathcal{U}$ by contractible open sets such that for every element $U$ of the cover and  every point $X_0\in U$ we have   $\mathcal{X}_{|U}\simeq \pi^{-1}(U)\simeq U\times X_0$, which implies that 
$X_u\simeq X_0$  for all $u\in U$  as  $C^{\infty}$ orbifolds; moreover, $H^k(X_u)\simeq H^k(X_0)$.   Thanks to the locally trivialization and  as quasi-smooth hypersurfaces  are   orbifolds \cite{bc}, we can put an orbifold structure on $\mathcal{X}=\pi^{-1}(U)$.

\paragraph{The Cartan-Lie formula.}
 For every $k$, let $\mathcal H^k$ be the complex vector bundle on $\mathcal U_\beta$ associated to the local system
 $R^k\pi_\ast\C$. 
Let $\Omega$  be a Zariski $k$-form on the orbifold $\mathcal{X}$ such that 
$\Omega_u = \Omega_{\vert X_u}$  is closed for every $u\in U$; we can associate with it a local section $\omega$  of the vector bundle $\mathcal{H}^{k}$
by letting $$\omega(u) = [\Omega_u] \in H^k(X_u,\C).$$

The following result computes the Gauss-Manin connection $\nabla: \mathcal{H}^k\to  \mathcal{H}^k\otimes \Omega_{\mathcal{U}}$ in the direction $w$ restricted to $X_0$.

\begin{prop}[Cartan-Lie Formula] If $w\in T_{\mathcal{U},X_0}$ and $v\in \Gamma (T_{\mathcal{X}|X_0})$ is such that $\phi_{*,x}(v)=w$ for all $ x\in X_0$, one has 
\begin{equation}
\nabla_w  (\omega) =\left[ \iota_v (d\Omega)_{|X_0} \right]     
\end{equation}
\end{prop}

\begin{proof}
See \cite{LiuZhuang}; actually the proof goes as in the classical case, see
Proposition 9.2.2 in \cite{v1}.
\end{proof}

Again we take $U$ a contractible open set trivializing $\mathcal{X_U}_{\vert U}\simeq U\times X_0$.
\begin{dfn} The period map 
$$\mathcal{P}^{p,k}: \mathcal{U}\to  \operatorname{Grass}(b^{p,k},H^k(X,\C)) $$
is the map which to $u\in {U}$ associates the term 
$F^pH^k(X_u,\C)$ in the Hodge filtration of $ H^k(X_u,\C)\simeq H^k(X_0,\C)$.
\end{dfn}

Here $b^{p,k} = \dim F^pH^k(X_u,\C)$.
Note that $\mathcal{P}^{p,k}$ is a map of complex manifolds.

\begin{prop} The period map $\mathcal{P}^{p,k}$ is holomorphic.
\end{prop}

\begin{proof} For the reader's convenience we sketch here a proof of this result, although it has been actually   already   proved in \cite{LiuZhuang}.
By Theorem 7.9 in \cite{k} and the fact that Hodge theorem holds also in the orbifold case (\cite{Saito,Zaf09} and also   section 2.1 in 
 \cite{l})
$\mathcal{P}^{p,k}$ is a $C^{\infty}$ map. The rest of the proof follows as in Theorem 10.9 in \cite{v1}, whose strategy is to prove that  the $\C$-linear extension of the  differential to $T_u{U}\otimes \C$ of $\mathcal{P}^{p,k}$ vanishes on the vectors of type (0,1).\end{proof}

\begin{rmk} There is an intrinsic relation between  the differential $$d\mathcal{P}^{p,k}_u(w)\colon F^pH^k(X_u)\to  H^k(X_0)/F^pH^k(X_u)$$ and the covariant derivative $\nabla_w\colon \mathcal{H}^k\to  \mathcal{H}^k$, namely, given  $\sigma\in F^pH^k(X_u) $ one can construct a local section of $\mathcal H^k$ over $U$
$$
\begin{array}{ccc}
\tilde{\sigma} \colon U&\to &   H^k(X_u) \\
u'&\mapsto& \tilde{\sigma}(u')\in F^pH(X_{u'}) 
\end{array}
$$
such that $\tilde{\sigma}(u)=\sigma $. Hence, 
$$d\mathcal{P}^{p,k}_u(w)(\sigma)=\nabla_w \tilde{\sigma} \ \text{mod}\  F^pH^k(X_u)$$

\end{rmk}

\begin{prop}[Griffiths Transversality]   
$$\nabla F^p \mathcal{H}^k\subset F^{p-1}\mathcal{H}^k $$
\end{prop}

\begin{proof}
By the Cartan-Lie formula and the above remark 
$$d\mathcal{P}^{p,k}_w(\sigma)=\left[\iota_v d\Omega_{|X_0} \right]\ \text{mod}\  F^pH^k(X_u).$$ The fact that $\mathcal{P}^{p,k}$ is holomorphic implies that that $\iota_vd\Omega_{|X_0}\in F^pH^k(X_u)$ if $v$ is of type $(0,1)$, so that  if $v$ is of type $(1,0)$ we get  $\iota_vd\Omega_{|X_0}\in F^{p-1}H^k(X_u)$.
\end{proof}

\begin{thm} Each $NL_{\lambda,U}^{k,\beta}\subset \mathcal{U}$ can be defined locally by   $h^{k-1,k+1}$ holomorphic equations, where $h^{k-1,k+1}=\operatorname{rk}  F^{k-1}\mathcal{H}^{2k}/F^k\mathcal{H}^{2k}.$

\end{thm}

\begin{proof}  Once  Griffiths Transversality has been generalized,    the proof goes as in classical case,
see  Lemma 3.1  in \cite{v} and section 5.3  in \cite{v2}.
\end{proof}

This proves Theorem \ref{uppercodim}. 
%
%
%
%
%

\end{document}